\documentclass{amsart}
\usepackage{stmaryrd}
\usepackage{silence}
\usepackage{upgreek}
\usepackage{amssymb}
\usepackage{faktor}
\usepackage[all,cmtip]{xy}
\usepackage{amsmath} 
\usepackage{mathrsfs}
\usepackage{tikz}
\usepackage{tikz-cd}
\usepackage{enumitem}
\usepackage{mathtools}
\usepackage[mathcal]{euscript}
\usepackage{graphicx}
\usepackage{url}
\usepackage{hyperref}
\usepackage[T1]{fontenc}
\usepackage{subfigure}

\usetikzlibrary{shapes.geometric}
\usetikzlibrary{decorations.markings}
\usetikzlibrary{patterns,decorations.pathreplacing}

\usepackage{ragged2e}

\newcommand{\id}{\mathrm{id}}

\newcommand{\Map}{\mathrm{Map}}

\newcommand{\pt}{\mathrm{pt}}

\DeclareMathOperator*{\colim}{\mathrm{colim}}

\definecolor{coloryellow}{RGB}{240,228,66}
\definecolor{colorskyblue}{RGB}{86,180,233}
\definecolor{colorvermillion}{RGB}{213,94,0}

\DeclareSymbolFont{sfletters}{OT1}{cmss}{m}{n}
\DeclareMathSymbol{\sTheta}{\mathord}{sfletters}{"02}

\theoremstyle{definition}
\newtheorem{definition}{Definition}[section]

\newtheorem{convention}[definition]{Convention}

\theoremstyle{plain}
\newtheorem{proposition}[definition]{Proposition}
\newtheorem{lemma}[definition]{Lemma}
\newtheorem{corollary}[definition]{Corollary}
\newtheorem{theorem}[definition]{Theorem}

\newtheorem*{bound}{Universal upper bound}

\theoremstyle{remark}
\newtheorem{remark}[definition]{Remark}

\makeatletter
\@addtoreset{definition}{section}
\makeatother

\usepackage{marginnote}
    \DeclareFontFamily{U}{wncy}{}
    \DeclareFontShape{U}{wncy}{m}{n}{<->wncyr10}{}
    \DeclareSymbolFont{mcy}{U}{wncy}{m}{n}
    \DeclareMathSymbol{\Sha}{\mathord}{mcy}{"58}

\newsavebox{\foobox}

\title{On the analog category of finite groups}
\author{Ben Knudsen and Shmuel Weinberger}

\begin{document}

\maketitle

\begin{abstract}
We show that the analog category of a finite group is essentially proportional to the size of its largest Sylow subgroup. We conclude that the universal upper bound given by the order of the group is very far from optimal.
\end{abstract}

\section{Introduction}

We continue the probabilistic reimagining of the foundations of topological robotics \cite{Farber:TCMP} begun simultaneously in \cite{KnudsenWeinberger:ACC} and \cite{DranishnikovJauhari:DTCLSC}, in which motion planning is conducted according to continuously varying probability measures on the relevant space of paths. The resulting ``analog'' invariants---which bound their classical counterparts, the Lusternik--Schnirelmann category and topological complexity, from below---display surprisingly subtle behavior. 

For example, the analog category of an aspherical space with torsion-free fundamental group is equal to the cohomological dimension of that group \cite[Thm. 1.1]{KnudsenWeinberger:ACC}, a direct analogue of the Eilenberg--Ganea theorem. Thus, in this case, analog category equals category. On the other hand, in the case of a finite fundamental group, the classical category is always infinite, while in our setting we have the following \cite[Thm. 7.2]{KnudsenWeinberger:ACC}---see Section \ref{section:analog category} for the definition of $\mathrm{acat}(G)$.

\begin{bound}
If $G$ is finite, then $\mathrm{acat}(G)+1\leq |G|$.\footnote{For stating our results, it is convenient to work with the ``unreduced'' category, which differs from the ``reduced'' convention by $1$. To avoid confusion, we will simply write $\mathrm{acat}(G)$ for the latter and $\mathrm{acat}(G)+1$ for the former. Both conventions are common throughout the literature on Lusternik--Schnirelmann category and topological complexity.}
\end{bound}

We show here that this bound is a very bad one in three senses: the groups for which it is sharp are highly constrained; the difference of the two sides is arbitrarily large; and their ratio is arbitrarily small. Provisionally, let us call $G$ $a$-\emph{special} if $|G|=ap^s$ with $p$ a prime, $(a,p)=1$, and $p^s>a$ (thus, a $1$-special group is simply a $p$-group).

\begin{theorem}\label{thm:failure}
In what follows, $G$ refers to a finite group.
\begin{enumerate}
\item If $G$ is not $a$-special for some $a\in \{1,2,3\}$, then the universal upper bound is strict for $G$.
\item If $G$ is $1$- or $2$-special, then the universal upper bound is sharp for $G$.\footnote{Unfortunately, we do not know whether the universal upper bound is sharp for $3$-special groups.}
\item For any $N\geq0$ and $\epsilon>0$, there exists a $G$ such that $\mathrm{acat}(G)+1\leq |G|-N$ and $\mathrm{acat}(G)+1\leq\epsilon |G|$.
\end{enumerate}
\end{theorem}

In particular, the universal upper bound is sharp for $p$-groups and strict for almost all other groups. For groups of the latter type, we show that the analog category is roughly proportional to the size of the largest Sylow subgroup.

\begin{theorem}\label{thm:range}
Let $G$ be a finite group not of prime power order. If $P\leq G$ is a Sylow subgroup of maximal order, then 
\[\min\left\{2, \textstyle{\frac{|N(P)|}{|P|}}\right\}\leq \frac{\mathrm{acat}(G)+1}{|P|}\leq 3,\]
where $N(P)$ denotes the normalizer of $P$ in $G$.
\end{theorem}

In more prosaic terms, the lower bound is $1$ when the largest Sylow subgroup of $G$ is self-normalizing\footnote{As shown in \cite{GuralnickMalleNavarro:SNSG}, the admission of a self-normalizing Sylow subgroup places strong constraints on a group.}, and otherwise it is $2$; for example, the lower bound of $2$ obtains for all nilpotent groups. 

Prior to our work here, the quantity $\mathrm{acat}(G)$ was almost completely unknown apart from the universal upper bound and a calculation for cyclic groups of prime order \cite{Dranishnikov:DTCG}. Strictly speaking, this last calculation was of an a priori different invariant, the distributional category; we show here that the two coincide for finite groups, a special case of \cite[Conj. 1.2]{KnudsenWeinberger:ACC}.

\begin{theorem}\label{thm:comparison}
For any finite group $G$, the analog and distributional category of $G$ coincide.
\end{theorem}

In fact, the same argument may be used to show that the analog and distributional versions of the $r$th sequential topological complexity $\mathrm{TC}_r$ coincide for every $r$.

\subsection{Conventions}
We write $\Delta^S$ for the simplex spanned by the set $S$. Thus, a point in $\Delta^S$ is a formal sum $\sum_{s\in S} t_ss$, where the \emph{barycentric coordinates} $t_s$ are non-negative numbers summing to $1$, all but finitely many of which vanish. In the case $S=\{1,\ldots, n\}$, we make the abbreviation $\Delta^{\{1,\ldots, n\}}=\Delta^{n-1}$; explicitly, this choice amounts to the following slightly non-standard notational convention:
\[\Delta^{n-1}:=\left\{(t_1,\ldots, t_n)\in[0,1]^n:\sum_{i=1}^nt_i=1\right\}.\] We work in Steenrod's convenient category of topological spaces \cite{Steenrod:CCTS}---see \cite[Appendix A]{KnudsenWeinberger:ACC} for a summary of relevant facts about these spaces. Topological spaces are implicitly convenient, as are limits, including products, and mapping spaces. Convenient colimits, when they exist, are the same as ordinary colimits. The adjective ``compact'' refers to the definition in terms of open covers. We write $BG$ for the classifying space of the (discrete) group $G$, i.e., the geometric realization of its nerve. We write $EG$ for the universal cover of $BG$ and $X^{hG}=\Map^G(EG, X)$ for the space of homotopy fixed points of the $G$-space $X$, where $\mathrm{Map}^G$ denotes the space of $G$-equivariant maps.

\section{The analog category of a group}\label{section:analog category}

The purpose of this section is to establish the following formula, which the reader may take as a definition of the analog category $\mathrm{acat}(G)$. We recall that $\Delta^G$ denotes the simplex spanned by the elements of $G$, and $\Delta^G_n$ is its $n$-skeleton.

\begin{proposition}\label{prop:hfp}
For any group $G$, we have
$\mathrm{acat}(G)=\min\{n\mid \left(\Delta_n^G\right)^{hG}\neq \varnothing\}.$
\end{proposition}

We begin with a brief review of the invariants of \cite{KnudsenWeinberger:ACC}, which are defined in terms of the set $\mathcal{P}(X)$ of probability measures with finite support on the topological space $X$. We view $\mathcal{P}(X)$ as a topological space with the quotient topology inherited from the various maps
\begin{align*}
X^n\times \Delta^{n-1}&\longrightarrow\mathcal{P}(X)\\
(x,t)&\mapsto \sum_{i=1}^n t_i \delta_{x_i}.
\end{align*} 
We write $\mathcal{P}_n(X)\subseteq \mathcal{P}(X)$ for the subspace of measures with support of cardinailty at most $n$. 

Classically, spaces of probability measures are often topologized using the L\'{e}vy--Prokhorov metric, which metrizes the topology of weak convergence when the background space $X$ is a separable metric space. We direct the reader to Section \ref{section:comparison} below for some comparisons between the two approaches. The twin advantages of ours are its generality, as we do not even require the background space to be metrizable, and its excellent technical features, which are summarized in the following result.

\begin{theorem}[{\cite[Thm. 2.7]{KnudsenWeinberger:ACC}}]\label{thm:point set}
The functor $\mathcal{P}$ is an endofunctor on the category of convenient spaces, which preserves homotopy, sifted colimits, quotient maps, and closed embeddings.
\end{theorem}

Given a map $f:X\to Y$, we may consider the space of probability measures on $X$ with fiberwise support over $Y$, namely 
\[\mathcal{P}(f)=\left\{\sum_{i=1}^nt_i\delta_{x_i}\in\mathcal{P}(X) : f(x_1)=f(x_2)=\cdots=f(x_n)\right\},\] and we set $\mathcal{P}_n(f)=\mathcal{P}_n(X)\cap \mathcal{P}(f)$. Sending an element of $\mathcal{P}(f)$ to the point in whose fiber it is supported defines a continuous map to $Y$.

\begin{definition}\label{def:analog}
The \emph{analog sectional category} of the map $f:X\to Y$ is the least $n$ such that $\mathcal{P}_{n+1}(f)\to Y$ admits a section. The \emph{analog category} of the space $X$, denoted $\mathrm{acat}(X)$, is the analog sectional category of the evaluation map $(X,x_0)^{([0,1],\{0\})}\to X$, where $x_0\in X$ is any basepoint.
\end{definition}

As our interest here lies solely in the aspherical context, we permit ourselves the abusive abbreviation $\mathrm{acat}(G)=\mathrm{acat}(BG)$.

For the proof of the proposition stated above, we require the following standard fact.

\begin{lemma}\label{lem:section space}
For any $G$-space $X$, the homotopy fixed point space $X^{hG}$ is canonically weakly equivalent to the space of sections of the canonical map $EG\times_G X\to BG$.
\end{lemma}
\begin{proof} Consider the following commutative diagram of mapping spaces: 
\[\xymatrix{
\Map^G(EG,EG\times X)\ar[r]\ar[d]&\Map^G(EG,EG\times_G X)\ar[d]\ar@{=}[r]^-\sim&\Map(BG, EG\times_GX)\ar[d]\\
\Map^G(EG,EG)\ar[r]&\Map^G(EG,BG)\ar@{=}[r]^-\sim&\Map(BG,BG)
}\] The section space in question is the fiber of the righthand vertical map over $\id_{BG}$, while $X^{hG}$ is the fiber of the lefthand vertical map over $\id_{EG}$. The claim follows after noting that the vertical maps are fibrations and the lefthand square a homotopy pullback.
\end{proof}

\begin{proof}[Proof of Proposition \ref{prop:hfp}]
It is well-known that, for any group $G$, there is a commutative diagram of topological spaces
\[\xymatrix{
G\ar[d]\ar[r]&(BG,x_0)^{([0,1],\{0,1\})}\ar[d]\\
EG\ar[d]_-\pi\ar[r]&(BG,x_0)^{([0,1],\{0\})}\ar[d]\\
BG\ar@{=}[r]&BG
}\] in which the horizontal arrows are homotopy equivalences and the vertical columns are (Hurewicz) fiber sequences. It follows from \cite[Cor. 5.3]{KnudsenWeinberger:ACC} that $\mathrm{acat}(G)$ is the analog sectional category of $\pi$, which is a fiber bundle with structure group $G$; therefore, by \cite[Cor. 5.8]{KnudsenWeinberger:ACC}, we have \[\mathcal{P}_{n+1}(\pi)\cong EG\times_G\mathcal{P}_{n+1}(G)\] as spaces over $BG$, and it is easy to see that $\mathcal{P}_{n+1}(G)\cong\Delta^G_n$ as $G$-spaces, so the claim follows from Lemma \ref{lem:section space}.
\end{proof}

\begin{corollary}\label{cor:contractible}
If $X$ is a contractible $G$-space with an equivariant map $X\to \Delta^G_n$, then $\mathrm{acat}(G)\leq n$.
\end{corollary}
\begin{proof}
Contractibility implies that $X^{hG}\neq \varnothing$, since an equivariant map $EG\to X$ may be constructed by elementary obstruction theory; alternatively, the map $EG\times_G X\to BG$ is a trivial fibration, since $EG\to \mathrm{pt}$ is a homotopy equivalence, and Lemma \ref{lem:section space} applies. It follows that $(\Delta^G_n)^{hG}$ receives a map from a non-empty space, hence is itself non-empty. The claim follows from Proposition \ref{prop:hfp}.
\end{proof}

\begin{remark}
Essentially the same argument shows that the $r$th analog topological complexity of $BG$, as defined in \cite{KnudsenWeinberger:ACC}, is equal to the least $n$ such that the $G^r$-space $\Delta_n^{G^r/G}$ admits a homotopy fixed point.
\end{remark}

\section{Designer complexes}
This section is concerned with the construction of certain contractible equivariant cell complexes, which, via Corollary \ref{cor:contractible} and obstruction theory, will be the key to proving our main results. The ideas here are mostly taken from the work of Assadi \cite{Assadi:FGASCMCWC}, following Oliver, Conner--Floyd, Smith, and others, but the specificity of our situation permits some simplification and hence a relatively self-contained account.

In what follows, the group $G$ is always finite. As a matter of terminology, we say that a space is $p$-acyclic if its mod $p$ reduced homology vanishes.

\begin{definition}
Let $X$ be a $G$-complex. Given a prime $p$, we say that $X$ is \emph{Smith $p$-acyclic} if $X^P$ is $p$-acyclic for every nontrivial $p$-subgroup $P$. We say that $X$ is \emph{Smith acyclic} if $X$ is Smith $p$-acyclic for every prime $p$.
\end{definition}

The relevance of this definition lies in its connection to obstruction theory.

\begin{proposition}\label{prop:smith projective}
Let $X$ be a $G$-complex of dimension $m$.
\begin{enumerate}
\item If $X$ is Smith ($p$-)acyclic, then a ($p$-)acyclic $G$-complex may be obtained from $X$ by attaching free cells of dimension at most $m+1$.
\item If $X$ is acyclic, then a contractible $G$-complex may be obtained by attaching free cells of dimension at most $3$.
\end{enumerate}
\end{proposition}
\begin{proof}
For the first claim, if $X$ is Smith acyclic, then \cite[Prop. I.1.6]{Assadi:FGASCMCWC} guarantees that we may achieve ayclicity below degree $m$ and $\mathbb{Z}[G]$-projectivity in degree $m$ by attaching free cells of dimension at most $m$. Thus, by the Eilenberg swindle, we may achieve acyclicity by further cell attachments of dimension $m$ and $m+1$. The Smith $p$-acyclic case is similar, invoking \cite[Lem. II.1.5]{Assadi:FGASCMCWC} instead, and obtaining instead a degree $m$ mod $p$ homology group projective over $\mathbb{F}_p[G]$. 

 For the second claim, note first that $X$ is path connected by acyclicity. We first kill the fundamental group of $X$ by attaching free $2$-cells indexed by a set $I$. Calling the resulting complex $Y$, the long exact sequence for the pair $(Y,X)$ shows that 
\[
\widetilde H_i(Y)\cong 
\begin{cases}
\bigoplus_I\mathbb{Z}[G]&\quad i=2\\
0&\quad \text{otherwise.}
\end{cases}
\] Thus, a simply connected acyclic $G$-complex may be obtained by attaching free $3$-cells, and any such complex is contractible by Whitehead's theorem.
\end{proof}

Our main construction will proceed inductively and one prime at a time. In order to state the main result, we require the following definition, which will form the basis for our induction.

\begin{definition}
Let $G$ be a finite group and $p$ a prime dividing $|G|$.
\begin{enumerate}
\item A subgroup $H\leq G$ is called a \emph{$p$-intersection} if it is an intersection of $p$-Sylow subgroups.
\item The \emph({$p$-Sylow) depth} of the $p$-intersection $H$ is the largest $d$ for which there is a chain $H=H_d<H_{d-1}<\cdots <H_1=P$ of proper inclusions with $P$ a $p$-Sylow subgroup and each $H_i$ a $p$-intersection.
\item The \emph{($p$-Sylow) depth} of a $p$-subgroup $H\leq G$ is the maximal depth of a $p$-intersection containing $H$.
\item The \emph{($p$-Sylow) depth} of $G$, denoted $d_p(G)$, is the maximal depth of a $p$-intersection in $G$.
\end{enumerate}
We adopt the convention that $d_p(G)=0$ if and only if $(p,|G|)=1$.
\end{definition}

It is easy to see that the depth of $H$ as a $p$-subgroup coincides with its depth as a $p$-intersection. It is also easy to see that the $d_p(G)$ is bounded above by the number of distinct $p$-Sylow subgroups of $G$, as well as by the exponent of $p$ in $|G|$.

\begin{lemma}\label{lem:rank goes down subgroup}
Let $H$ be a $p$-intersection of depth $d$ and $K$ any $p$-subgroup containing $H$. The depth of $K$ is at most $d$, with equality if and only if $H=K$.
\end{lemma}
\begin{proof}
We may assume that $H\neq K$. Supposing that $K$ has depth $s\geq d$, we obtain the chain of inclusions 
\[H<K\leq H_s<H_{s-1}<\cdots<H_1=P,\]
implying that the depth of $H$ is at least $s+1>d$, a contradiction.
\end{proof}

\begin{corollary}\label{cor:containment}
Let $H_1$ and $H_2$ be $p$-intersections. If $H_1< H_2$, then the depth of $H_1$ is greater than the depth of $H_2$.
\end{corollary}

\begin{corollary}\label{cor:uniqueness}
If $K$ is a $p$-subgroup of depth $d$, then $K$ is contained in a unique $p$-intersection of depth $d$.
\end{corollary}
\begin{proof}
Let $H_1\neq H_2$ be $p$-intersections of depth $d$ containing $K$. Then $H_1\cap H_2$ is a $p$-intersection properly contained in $H_1$, hence of strictly greater depth by Corollary \ref{cor:containment}. It follows that $K$ has depth greater than $d$, a contradiction.
\end{proof}

We write $\mathcal{I}_d=\mathcal{I}_d(p)$ for the set of $p$-intersections (in $G$, implicitly) of depth $d$, regarded as a $G$-set under conjugation.

\begin{convention}
For the remainder of this section, we assume that $G$ is a finite group not of prime power order.
\end{convention}

We come now to the main construction (compare \cite[Thm. II.1.4]{Assadi:FGASCMCWC}).

\begin{theorem}\label{thm:one prime}
Let $p$ be a prime dividing $|G|$. For $0\leq d\leq d_p(G)$, there are $G$-complexes $X_p(G)_d$ with the following properties:
\begin{enumerate}
\item $X_p(G)_0$ is free of dimension $1$;
\item $X_p(G)_{d+1}$ is obtained from $X_p(G)_{d}$ by attaching cells of dimension at most $d+2$ with isotropy in $\mathcal{I}_{d+1}$;
\item $X_p(G)_d^P$ is $p$-acyclic for every nontrivial $p$-subgroup $P$ of depth at most $d$.
\end{enumerate}
\end{theorem}

\begin{lemma}\label{lem:induction fixed points}
Let $H\leq G$ be a subgroup and $X$ an $N(H)$-space. For any $g\in G$, there is a canonical homeomorphism
\[\left(G\times_{N(H)} X\right)^{gHg^{-1}}\cong X^H.\]
\end{lemma}
\begin{proof}
Consider the standard decomposition $G\times_{N(H)} X=\bigsqcup_{[g_i]\in G/N(H)} g_iX$, and take $g$ to be one of our coset representatives. Writing $ghg^{-1}g_i=g_jh'$, we have $ghg^{-1}\cdot g_ix=g_j(h'\cdot x)$. It follows that $g_ix$ is fixed by $gHg^{-1}$ if and only if $x$ is fixed by $H$ and $[g]=[g_i]$, which is to say $g=g_i$. Thus, the desired homeomorphism is given (from right to left) by $x\mapsto gx$.
\end{proof}

\begin{proof}[Proof of Theorem \ref{thm:one prime}]
We proceed by simultaneous induction on $d$ and $d_p(G)$. In the case $d=d_p(G)=0$, we let $X_p(G)_0$ be any connected $1$-dimensional free $G$-complex, e.g., a Cayley graph. Notice that the third condition is vacuous in this case. For $d=0$ and general $G$, we choose a $p$-Sylow subgroup $P$ and set
\[X_p(G)_0=G\times_{N(P)} X_p(N(P)/P)_0.\]

For $d=1$, let $P$ be as above and consider $X_p(N(P)/P)_0$. Since $P$ is $p$-Sylow, the group ring $\mathbb{F}_p[N(P)/P]$ is semisimple by Maschke's theorem. It follows that $\widetilde H_1(X_p(N(P)/P)_0;\mathbb{F}_p)$ is projective over $\mathbb{F}_p[N(P)/P]$; therefore, by Proposition \ref{prop:smith projective} and the Eilenberg swindle, we may attach free cells of dimension $1$ and $2$ to obtain a $p$-acyclic $N(P)/P$-complex $\overline{X}_p(N(P)/P)_0$. Finally, we define 
\[X_p(G)_1=G\times_{N(P)}\overline{X}_p(N(P)/P)_{0}.\] The second property holds by construction, and the third follows from the observation that, by Lemma \ref{lem:induction fixed points}, the fixed set of any $p$-Sylow subgroup is homeomorphic to $\overline{X}_p(N(P)/P)_0$, which is $p$-acyclic by construction.

In the general case, choose a $p$-intersection $H\leq G$ of depth $d+1$ and consider the $N(H)/H$-space $X_p(G)_d^H$. We claim that this space is Smith $p$-acyclic; indeed, given a nontrivial $p$-subgroup $P\leq N(H)/H$, we have $(X_p(G)_d^H)^P=X_p(G)_d^{\widetilde P}$, where $H\leq \widetilde P$ is the subgroup of $N(H)$ corresponding to $P$, and the depth of $\widetilde P$ is at most $d$ by Lemma \ref{lem:rank goes down subgroup}, since $P$ was assumed nontrivial, so the claim follows by induction. Therefore, by Proposition \ref{prop:smith projective}, we may achieve $p$-acyclicity after attaching free $N(H)/H$-cells of dimension at most $d+2$, and, indexing these cells by $i\in I$, we achieve the same result $G$-equivariantly for all conjugates of $H$ at once via the construction
\[G\times_{N(H)}\left(\bigsqcup_{i\in I}N(H)/H\times D^{n_i}\right)\bigsqcup_{G\times \bigsqcup_{i\in I}N(H)/H\times S^{n_i-1}}X_p(G)_d.\]
By Corollary \ref{cor:containment}, this construction does not alter the fixed set of any member of $\mathcal{I}_{d+1}$ not conjugate to $H$; therefore, we may define $X_p(G)_{d+1}$ to be the result of iterating the construction over $\mathcal{I}_{d+1}/G$. 

The second condition holds by construction. To check the third, we note that, if $P$ has depth less than $d$, then $X_p(G)_d^P=X_p(G)_{d-1}^P$ by construction, since $P$ is contained in no member of $\mathcal{I}_d$ by definition, and the latter is $p$-acyclic by induction. On the other hand, if $P$ has depth $d$, then $P$ is contained in a \emph{unique} $H\in \mathcal{I}_d$ by Corollary \ref{cor:uniqueness}, so $X_p(G)_d^P=X_p(G)_d^H$, which was constructed to be $p$-acyclic.
\end{proof}

We write $X_p(G)=X_p(G)_{d_p(G)}$.

\begin{corollary}
There is a $G$-complex $X(G)$ with the following properties:
\begin{enumerate}
\item $X(G)$ is obtained from $\bigsqcup_{p\mid |G|} X_p(G)$ by attaching free cells of dimension at most $\max_p d_p(G)+2$
\item $X(G)$ is contractible.
\end{enumerate}
\end{corollary}
\begin{proof}
By construction, given $p\neq q$ dividing $|G|$, every $p$-subgroup of $G$ acts without fixed points on $X_q(G)$, so the third condition of Theorem \ref{thm:one prime} implies that the disjoint union in question is Smith acyclic. Since the dimension of $X_p(G)$ is $d_p(G)+1$, and since $\max_p d_p(G)+2\geq 3$, Proposition \ref{prop:smith projective} shows that we may achieve first acyclicity, then contractibility after the indicated type of cell attachment.
\end{proof}

\section{Proofs of the main results}

Our strategy will be to exploit Corollary \ref{cor:contractible} by applying obstruction theory to the complex $X(G)$ constructed in the previous section. In order to proceed, we require information on the connectivity of fixed point sets.

\begin{proposition}\label{prop:skeleton fixed points}
For any $H\leq G$ and any $n\in\mathbb{Z}$, there is a canonical $N(H)/H$-equivariant homeomorphism
\[\left(\Delta^G_n\right)^H\cong \Delta^{G/H}_k,\]
where $k={\lfloor\frac{n+1}{|H|}\rfloor-1}$.
\end{proposition}

\begin{corollary}\label{cor:connectivity}
For any $H\leq G$ and $n<|G|$, the connectivity of $(\Delta_n^G)^H$ is exactly $\lfloor \frac{n+1}{|H|}\rfloor-2$.
\end{corollary}

\begin{proof}[Proof of Proposition \ref{prop:skeleton fixed points}]
Since $H$ is finite, we may define a function $f:\Delta^{G/H}\to \Delta^G$ by the formula
\[f(t)_g=\frac{1}{|H|} t_{gH}.\] 
As the restriction of a linear map, this function is continuous, and its image is $H$-fixed by inspection; thus, we may we view $f$ as a map to $(\Delta^G)^H$. As such, it is $N(H)/H$-equivariant and injective by inspection, and we claim that it is also surjective. To see why, note that a point in $\Delta^G$ is fixed by $H$ if and only if the barycentric coordinate of $gh$ is independent of $h\in H$ for every $g\in G$. Thus, given $t\in (\Delta^G)^H$, setting $t_{gH}= |H|t_g$ defines an element of $f^{-1}(t)$. The homeomorphism $(\Delta^G)^H\cong \Delta^{G/H}$ follows, since both sides are compact and Hausdorff.

Now, a point of $\Delta^G$ lies in $\Delta_n^G$ if and only if at most $n+1$ of its barycentric coordinates are nonzero. We conclude that $f$ identifies $(\Delta_n^G)^H$ with the subspace of $\Delta^{G/H}$ in which at most $\frac{n+1}{|H|}$ barycentric coordinates are nonzero, as desired.
\end{proof}

In what follows, we write $X_p^\wedge$ for the completion of the space $X$ at the prime $p$---see \cite{MayPonto:MCAT}, for example. We recall that, according to one of several results known collectively as the ``generalized Sullivan conjecture,'' due to Carlsson \cite[Thm. B(c)]{Carlsson:ESHSC} and Dwyer--Miller--Neisendorfer \cite{DwyerMillerNeisendorfer:FCUASS}, the natural map $(X^P)_p^\wedge\to (X_p^\wedge)^{hP}$ is a weak equivalence for any $p$-group $P$ and finite dimensional $P$-CW complex $X$.


\begin{proposition}\label{prop:p-sub}
For any $p$-subgroup $P\leq G$, we have $\mathrm{acat}(G)\geq |P|-1$. If $P$ is not self-normalizing, then $\mathrm{acat}(G)\geq 2|P|-1$.
\end{proposition}

\begin{lemma}\label{lem:self-normalizing lemma}
Let $H\leq G$ be any subgroup. If $H$ is not self-normalizing, then $(G/H)^{hN(H)/H}=\varnothing$.
\end{lemma}
\begin{proof}
Our assumptions imply that $N(H)/H$ is a nontrivial group acting without fixed point on the discrete space $G/H$. An easy exercise shows that such a space admits no equivariant map from any connected $N(H)/H$-space with nontrivial action, hence no homotopy fixed point.
\end{proof}

\begin{proof}[Proof of Proposition \ref{prop:p-sub}]
We begin with a few elementary observations regarding the following commutative diagram of canonical maps, to which we will appeal throughout the argument:
{\tiny\[\xymatrix{
\Map(\pt,\Delta^G_n)^P\ar@{=}[d]\ar[r]&\left(\Map(\pt,\Delta_n^G)^P\right)_p^\wedge\ar@{=}[d]\ar[r]&\left(\Map(EG,(\Delta_n^G)_p^\wedge\right)^P\ar[d] &\Map(EG,\Delta_n^G)^P\ar[d]\ar[l] \\
(\Delta_n^G)^P\ar[r]& ((\Delta_n^G)^P)^\wedge_p\ar[r]&((\Delta_n^G)^\wedge_p)^{hP}& (\Delta_n^G)^{hP}\ar[l]
}\]}First, since the canonical map $EP\to EG$ is a $P$-equivariant homotopy equivalence, the third and fourth vertical arrows are weak equivalences. Second, since $N(P)/P$ acts canonically on the $P$-fixed set of any $G$-space, the arrows in the top row are all equivariant maps between $N(P)/P$-spaces. Third, by the Sullivan conjecture, the second map in the bottom row is a weak equivalence, and hence in the top row as well.

For the first claim of the proposition, it suffices to show that $(\Delta_n^G)^{hP}=\varnothing$ for $n<|P|-1$, since $(\Delta_n^G)^{hG}$ has a canonical map to this space. In this range, we have $(\Delta^G_n)^P=\varnothing$ by Proposition \ref{prop:skeleton fixed points}; in particular, this space is $p$-complete, so the first map in this row is also a weak equivalence (in fact, an equality, but this is irrelevant). We conclude that the target of the rightmost map is empty, so its source must be so as well.

For the second claim, it suffices as before to show that $(\Delta_n^G)^{hN(P)}=\varnothing$ for $|P|-1\leq n < 2|P|-1$. In this range, Proposition \ref{prop:skeleton fixed points} instead identifies $(\Delta_n^G)^P$ with the discrete $N(P)/P$-space $G/P$, which is also $p$-complete. As before, it follows that the first arrow in the bottom row of the diagram above is a weak equivalence, and hence in the top row as well. We also conclude from Lemma \ref{lem:self-normalizing lemma} that $(\Delta_n^G)^P$ has no homotopy fixed points for the action of $N(P)/P$. Since the homotopy fixed points functor preserves weak equivalences, a diagram chase as in the previous paragraph yields the conclusion that 
\[\left(\Map(EG,\Delta_n^G)^P\right)^{hN(P)/P}=\varnothing.\]
An easy exercise shows that this space is weakly equivalent to $(\Delta_n^G)^{hN(P)}$, and the claim follows.
\end{proof}

\begin{proof}[Proof of Theorem \ref{thm:range}]
The lower bound follows from Proposition \ref{prop:p-sub}. For the upper bound, it suffices by Corollary \ref{cor:contractible} to construct an equivariant map $X(G)\to \Delta_n^G$ for $n\geq 3q-1$, where $q$ is the largest prime power dividing $|G|$. To begin, since $X_p(G)_0$ is free of dimension $1$ for each prime $p$, there is no obstruction to constructing a map to $\Delta_n^G$ provided $n\geq 1$, which certainly holds in our situation. Proceeding inductively, we may extend an equivariant map from $X_p(G)_{d}$ to $X_p(G)_{d+1}$ provided $(\Delta_n^G)^H$ is $(d+1)$-connected for every $H\in\mathcal{I}_{d+1}(p)$; we will consider this question presently. Finally, equivariant maps from the various $X_p(G)$ may be extended to $X(G)$ provided $n\geq \max_p d_p(G)+2$, which certainly holds in our situation, since $d_p(G)$ is bounded above by the exponent of $p$ in $|G|$. 

Now, for $H\in \mathcal{I}_{d+1}(p)$, the connectivity of $(\Delta_n^G)^H$ is $\lfloor \frac{n+1}{|H|}\rfloor-2$ by Corollary \ref{cor:connectivity}, and $|H|\leq p^{s-d}$, where $p^s$ is the largest power of $p$ dividing $|G|$. Thus, it suffices to establish the inequality
\[1+\frac{d}{3}\leq \frac{q}{p^{s-d}}\] for every prime $p$ dividing $|G|$. Since $q\geq p^s$ by definition, the claim follows from the obvious inequality
\[1+\frac{d}{3}\leq p^d.\]
\end{proof}

\begin{proof}[Proof of Theorem \ref{thm:failure}]
Writing $|G|=qr$ with $q$ as above and appealing to Theorem \ref{thm:range}, we obtain the inequality
\[\frac{\mathrm{acat}(G)+1}{|G|}\leq \frac{3}{r}.\] If $r\notin\{1,2,3\}$, then the righthand side of this inequality is strictly less than $1$, implying the first claim. The second claim is immediate from Proposition \ref{prop:p-sub} and the universal upper bound. For the third claim, fix $N$ and $\epsilon$ and choose distinct primes $p>q$ and a natural number $s$ so that $N\leq (q-3)p^s$ and $\epsilon\leq \frac{3}{q}$. In this case, the desired bounds hold for the cyclic group of order $p^sq$. 
\end{proof}

\section{Analog vs. distributional}\label{section:comparison}

The goal of this section is to prove Theorem \ref{thm:comparison}, claiming that the analog category of the finite group $G$ coincides with its distributional category in the sense of \cite{DranishnikovJauhari:DTCLSC, Jauhari:OSVDTC}. We begin by recalling the relevant definitions.\footnote{For the sake of easier reading, we depart from the notation of \cite{DranishnikovJauhari:DTCLSC}.}

Throughout, we will use the subscript $\mathrm{LP}$ to refer to the L\'{e}vy--Prokhorov metric; thus, we have $\mathcal{P}_n(X)_{\mathrm{LP}}$ for $X$ metric, and we have $\mathcal{P}_n(f)_{\mathrm{LP}}$ for continuous $f$ with metric source. The rule for turning a definition of an analog invariant into that of a distributional invariant is to add this subscript; thus, the distributional sectional category of $f:X\to Y$ with metric source is the least $n$ for which $\mathcal{P}_{n+1}(f)_{\mathrm{LP}}\to Y$ admits a section, and this definition specializes to the definition of distributional category as in Definition \ref{def:analog}.

\begin{remark}
An issue deserving of care is that the distributional category is not defined if $X$ is not metrizable---for example, $X=BG$ with $G$ infinite. One possible workaround is to appeal to the fact that any CW complex is metrizable up to homotopy equivalence \cite{Cauty:RDLES}. Fortunately, since we confine our discussion here to finite groups, the issue does not arise.
\end{remark}

Our main technical result comparing these notions is the following.

\begin{proposition}\label{prop:comparison}
Let $f:X\to Y$ be a map with $X$ metric and $Y$ convenient. If $f$ is proper, then the analog and distributional sectional category of $f$ coincide.
\end{proposition}

For the proof, we require the following.

\begin{lemma}\label{lem:identity}
For a metric space $X$, the identity function $\mathcal{P}_n(X)\to \mathcal{P}_n(X)_{\mathrm{LP}}$ is continuous for every finite $n\geq0$. If $X$ is compact, then each of these maps is a homeomorphism.
\end{lemma}

Note that the first claim of Lemma \ref{lem:identity} is simply the claim that the quotient topology on $\mathcal{P}(X)$ is finer than the L\'{e}vy--Prokhorov topology when both are defined, which can also be seen by considering an explicit basis for the latter \cite[\S 8]{Jauhari:OSVDTC}.

\begin{proof}[Proof of Lemma \ref{lem:identity}]
For the first claim, let $X$ be a metric space, and consider the commutative diagram
\[\xymatrix{
\colim_{K} \mathcal{P}(K)\ar[r]\ar[d]&\colim_{K} \mathcal{P}(K)_{\mathrm{LP}}\ar[d]\\
\mathcal{P}(X)\ar[r]&\mathcal{P}(X)_{\mathrm{LP}},
}\] where $K$ ranges over compact subspaces of $X$. The collection of such is filtered, hence sifted, so the lefthand arrow is a homeomorphism by Theorem \ref{thm:point set}. Thus, in order to establish continuity of the bottom arrow, it suffices to establish continuity of the top arrow; in other words, we may assume that $X$ itself is compact. From the definition of $\mathcal{P}(X)$, continuity is equivalent to continuity of each of the maps 
\begin{align*}
X^n\times \Delta^{n-1}&\longrightarrow\mathcal{P}(X)_{\mathrm{LP}}\\
(x,k)&\mapsto \sum_{i=1}^n t_i \delta_{x_i}.
\end{align*} which is to say sequential continuity, since the source is metric. Since $X$ is compact, hence separable, the topology on the target is the topology of weak convergence of measures, so sequential continuity follows from continuity of the composite
\[X^n\times\Delta^{n-1}\xrightarrow{f^n\times \iota}\mathbb{R}^n\times\mathbb{R}^{n}\xrightarrow{\langle-,-\rangle}\mathbb{R},\] where $\iota$ is the inclusion and $f:X\to \mathbb{R}$ is an arbitrary continuous function.


For the second claim, if $X$ is compact, then so is $\mathcal{P}_n(X)$. Since $\mathcal{P}_n(X)_{\mathrm{LP}}$, as a metric space, is Hausdorff, and since the map in question is a continuous bijection, the claim follows.
\end{proof}

\begin{proof}[Proof of Proposition \ref{prop:comparison}]
It follows from Lemma \ref{lem:identity} that the top arrow in the commutative diagram
\[
\xymatrix{
\mathcal{P}_n(f)\ar[rr]^-{\id}\ar[dr]&&\mathcal{P}_n(f)_{\mathrm{LP}}\ar[dl]\\
&Y
}
\] is continuous, so a section of the lefthand map determines a section of the right. Thus, the distributional sectional category bounds the analog from below. For the reverse inequality, we will show that any section $\sigma$ of the righthand map is continuous when considered as a map to $\mathcal{P}_n(f)$. By assumption, the space $Y$ is the colimit of its compact subsets; therefore, since $\sigma|_K$ factors through $\mathcal{P}_n(f|_{f^{-1}(K)})_{\mathrm{LP}}$, we may assume without loss of generality that $Y$ itself is compact. In this case, since $f$ is proper, it follows that $X$ is also compact, and Lemma \ref{lem:identity} implies the claim.
\end{proof}

\begin{lemma}\label{lem:lp homotopy}
For any $n\geq0$, the functor $\mathcal{P}_n(-)_\mathrm{LP}$ preserves homotopy, hence homotopy equivalence.
\end{lemma}
\begin{proof}
Using the topological basis given in \cite[\S3.1]{DranishnikovJauhari:DTCLSC}, it is easy to check that, for any metric space $Y$, the assignment 
\[\left(\sum_{i=1}^mt_i\delta_{x_i}, y\right)\mapsto \sum_{i=1}^mt_i\delta_{(x_i,y)} \] defines a continuous map $\mathcal{P}_n(X)_\mathrm{LP}\times Y\to \mathcal{P}_n(X\times Y)_\mathrm{LP}$. The claim follows easily after taking $Y=[0,1]$.
\end{proof}

\begin{proof}[Proof of Theorem \ref{thm:comparison}]
It suffices to show that the distributional category of $BG$ is the distributional sectional category of the map $\pi:EG\to BG$; indeed, the corresponding analog statement is true, as shown in the course of proving Proposition \ref{prop:hfp}, and $\pi$ is proper, so Proposition \ref{prop:comparison} applies. Considering the diagram
\[\xymatrix{
G\ar[d]\ar[r]&(BG,x_0)^{([0,1],\{0,1\})}\ar[d]\\
EG\ar[d]_-\pi\ar[r]&(BG,x_0)^{([0,1],\{0\})}\ar[d]\\
BG\ar@{=}[r]&BG
}\] from the proof of Proposition \ref{prop:hfp}, the claim follows by noting that the construction $f\mapsto \mathcal{P}_n(f)_{\mathrm{LP}}$ preserves Hurewicz fibrations by \cite[Prop. 5.1]{DranishnikovJauhari:DTCLSC} and that the construction $X\mapsto \mathcal{P}(X)_\mathrm{LP}$ preserves homotopy equivalence by Lemma \ref{lem:lp homotopy}.
\end{proof}

\begin{remark}
An alternative argument establishing that the distributional sectional category is a homotopy invariant of fibrations is given in \cite[Prop. 5.3]{Jauhari:OSVDTC}.
\end{remark}

\bibliographystyle{amsalpha}
\bibliography{references}

\end{document}